\documentclass[12pt]{article}
\usepackage {amssymb,mathrsfs,amsfonts}
\newtheorem{lemma}{Lemma}[section]
\newtheorem{theorem}[lemma]{Theorem}

\newtheorem{Example}[lemma]{Example}

\newcommand{\fp}{\hspace*{\fill} $\Box$}

\makeatletter \@addtoreset{equation}{section} \makeatother
 \date{}

\begin{document}

\makeatother\title{Complex Dunkl operators
\thanks{2000 Mathematics subject Classification: Primary 20F55;
Secondary 32A30, 33C52.}}
\author{Guangbin  Ren
\\
  {\small University of Science and Technology of China,
 Department}
\\ {\small of Mathematics, Hefei, Anhui 230026, P. R. China}
\\{\small E-mail:  rengb@ustc.edu.cn }
\and Helmuth R. Malonek
\\ {\small  Universidade de Aveiro, Departamento de Matem\'atica}
\\ {\small  P-3810-193 Aveiro, Portugal}
\\ {\small E-mail:  hrmalon@mat.ua.pt}}
 \maketitle

 \maketitle
\markboth {}{ {\it Ren} and {\it Malonek} }

\normalsize
\begin{abstract}
The real theory of the Dunkl operators   has  been developed very
extensively, while there still lacks the corresponding complex
theory. In this paper we introduce
 the complex Dunkl
operators for certain Coxeter groups. These complex Dunkl operators
have the  commutative property, which makes it possible to establish
a corresponding  complex  analysis of Dunkl operators.

\par \vskip0.3cm
 {\bf Keywords:} Dunkl operators, complex Coxeter group
\end{abstract}

\normalsize
 \baselineskip18pt
\par


\section{Introduction}

In the real Euclidean spaces, Dunkl extended the classical real
analysis by introducing the generalized differential operators
related to Coxeter groups \cite{{dunk2},{dunkl}}. The resulting
operators are  a family of
 commutative
differential-difference operators $\mathcal T_j$, called Dunkl
operators, which can be considered as perturbations of the usual
partial derivatives by reflection parts. These operators step from
the analysis of quantum many body system of
Calogero-Moser-Sutherland type \cite{DV} in mathematical physics.
They  also have root in the theory of special functions of several
variables, motivated from the theory of Riemann symmetric spaces,
whose spherical functions can be written as multi-variable
parametrized special functions. We refer to \cite{CKR, dunklx, Jeu,
Ros, SO} for the theory of special functions, Fourier  transforms,
Segal-Bargmann transforms,  and Clifford analysis in the Dunkl
setting.

  The purpose of the article is to
 establish the basis for the complex Dunkl
analysis  by introducing the complex Dunkl operators for certain
Coxeter groups and establishing their important property of
 commutativity.
We mention that Dunkl and Opdam \cite{DO} also introduced the  Dunkl
operators for complex reflection groups. But their operators are
generalization of the real partial differentiate operators, other
than the generalization of the complex operator
$\frac{\partial}{\partial z}$.

\section{Real Dunkl operators}

In this section, we recall the definition of the  real Dunkl
operators.

In $\mathbb R^N$, we consider the standard inner product
$$\langle x,
y\rangle=\sum_{j=1}^N x_j y_j$$ and the norm  $||x||=\langle x,
x\rangle^{1/2}$.

For a row vector $u\in\mathbb R^N\setminus\{0\}$, the reflection in
the hyperplane $\langle u \rangle^{\perp}$ orthogonal to $u$ is
defined by
$$\sigma_u(x)= x\sigma_u:=x-2\frac{\langle x, u\rangle}{||u||^2} u.$$

A root system is a finite set $R$ of nonzero vectors in $\mathbb
R^N$ such that
$$\sigma_u(R)=R,\qquad \textrm{and}\qquad
R\cap\mathbb Ru=\{\pm u\}$$ for any $u\in R$.

A positive subsystem $R_+$ is any subset of $R$ satisfying $R =
R_+\cup \{-R_+\}$. For example, for any $u_0\in\mathbb R^N$ such
that $\langle u, u_0\rangle\neq 0$ for all $u\in R$, then
$R_+:=\{u\in R: \langle u, u_0\rangle>0\}$ is a positive subsystem.

The real Coxeter group $W=W(R)$ (or real finite reflection group)
generated by the root system $R\subset \mathbb R^N$ is the subgroup
of orthogonal group $O(N)$ generated by $\{\sigma_u: u\in R\}$.

A multiplicity function on $R$ is a complex-valued function
\begin{eqnarray*}
\kappa: R &\longrightarrow \mathbb C\\
v&\mapsto \kappa_v
\end{eqnarray*}
which is invariant under the Coxeter group, i.e.,
$$\kappa_u=\kappa_{ug}, \qquad \forall\ u\in
R, \quad \forall\ g\in W.$$

The Dunkl operator $\mathcal T_j$, associated with the Coxeter group
$W(R)$ and the multiplicity function $\kappa$, is the first order
differential-difference operator:

\begin{eqnarray*}\mathcal{T}_j p(x)&=& \frac{\partial p(x)}{\partial x_j}
+\sum_{v\in R_+}\kappa_v \frac{p(x)-p(x\sigma_v)}{\langle
x,v\rangle}v_j
 \end{eqnarray*}
for any polynomials $p$ in $\mathbb R^N$.

The Dunkl operator $\mathcal T_j$ is a  homogeneous differential
operator of degree $-1$. By the $W$-invariance of the multiplicity
function $\kappa$, we have
$$
g^{-1}\circ \mathcal T_u \circ g = \mathcal T_{ug},\qquad \forall\
g\in W(R), \quad u\in \mathbb R^N,$$ where
$$\mathcal T_u=\sum_{j=1}^N u_j \mathcal T_j.$$

The remarkable property of the Dunkl operators is that the family
$\{\mathcal T_u, u\in\mathbb R^N\}$ generates a commutative algebra
of linear operators on the $\mathbb C$-algebra of polynomial
functions on $\mathbb R^N$.

\begin{Example} In the one-dimensional case $N=1$,  the root system $R$ is of type
$A_1$ (see \cite{Hum}), the reflection group $W=\mathbb Z_2$, and
the multiplicity function is given by a single parameter
$\kappa\in\mathbb C$. The Dunkl operator is  given  by
$$
\begin{array}{lcl}
\mathcal Tf(x)&=&f'(x)+\kappa\displaystyle\frac{f(x)-f(-x)}{x}.
\end{array}$$
\end{Example}

\section{Complex Dunkl operators}

\textbf{Assumption A:} Under  the standard embedding $\mathbb
R^N\subset \mathbb C^N$, we  assume that $R$ is a root system in
$\mathbb C^N\cong \mathbb R^{2N}$ satisfying
$$R\subset \mathbb R^N.$$

\bigskip

Let $R$ be a root system in $\mathbb C^N$  satisfying Assumption A.
Then $R\subset \mathbb R^N$,  so that we can define the positive
subsystem  $R_+\subset \mathbb R^N\subset \mathbb C^N $.

For any operator $A\in O(N)$,   we can regard $A$ as a real
orthogonal matrix under a fixed basis $\{e_\alpha\}$ of $\mathbb
R^N$. Therefore, $A$ is also a unitary matrix in $\mathbb C^N$.
Under the same basis $\{e_\alpha\}$ of $\mathbb C^N$, the matrix $A$
can thus be regarded as a unitary operator in $\mathbb C^N$.
Therefore the matrix point of view results in  a standard embedding
$$O(N)\hookrightarrow U(N),$$
where $U(N)$ is the unitary group.

In $\mathbb C^N$, we will use the same notation $\langle \cdot,
\cdot\rangle$ for the  extension of the Euclidean inner product,
i.e.,  $$\langle z, w\rangle=\sum_{j=1}^N z_j \overline{w_j},\qquad
\forall\ z, w\in \mathbb C^N,$$ and denote the norm  $||z||=\langle
z, z\rangle^{1/2}$. We also denote the bilinear extension of the
Euclidean inner product by $$(z, w)=\sum_{j=1}^N z_j {w_j}.$$

Let
$$v\in R\subset\mathbb R^N\hookrightarrow\mathbb C^N$$ and define  the
reflection $\sigma_v\in U(N)$ as matrix by
$$(\sigma_v)_{ij}=\delta_{ij}-2\frac{v_iv_j}{||v||^2},$$ so that
$$\sigma_v(z)= z\sigma_v:=z-2\frac{\langle z, v\rangle}{||v||^2} v=z-2\frac{( z, v)}{||v||^2} v.$$

The complex Coxeter group $G=G(R)$ is the subgroup of $U(N)$
generated by $\{\sigma_v: v\in R\}$, namely,
$$G:=\langle\sigma_v\in U(N): v\in R\subset\mathbb R^N \hookrightarrow  \mathbb C^N \rangle.$$
Consequently, for our choice of the root system, we have
$$G\subset O(N)\hookrightarrow U(N).$$

A multiplicity function $\kappa$ is   a $G$-invariant complex-valued
function on $R$.

Let $R\subset\mathbb R^N$ be a root system  and $R_+$ be  a positive
subsystem. Now we define the complex Dunkl operator as the first
order differential-difference operators in coordinate form for $1\le
j\le N$  by
\begin{eqnarray*}\mathcal{D}_j p(z)&=& \frac{\partial p(z)}{\partial z_j}
+\sum_{v\in R_+}\kappa_v \frac{p(z)-p(z\sigma_v)}{\langle
z,v\rangle}v_j
 \end{eqnarray*}
for any holomorphic polynomials $p$ in $\mathbb C^N$.

To  understand the second summand, notice that
$$\frac{dp}{dt}=\sum_{j=1}^N \frac{dp}{dz_j} \frac{dz_j}{dt}+\sum_{j=1}^N
\frac{dp}{d\overline z_j} \frac{d\overline z_j}{dt} \quad {\mbox
{and} }\quad  \frac{dp}{d\overline z_j}=0,$$ we have
\begin{eqnarray*}\label{eq2.1}
\frac{p(z)-p(z\sigma_v)}{\left<z,v\right>}&=-\frac{1}{\left<z,v\right>}\int_0^1
\frac{dp}{dt}(t(z\sigma_v)+(1-t)z)dt
\\
&= \int_0^1 \left<\nabla p(t(z\sigma_v)+(1-t)z), \frac{2v}{||v||^2}
\right>dt,
\end{eqnarray*} where $\nabla$ is the complex gradient
$$\nabla p(z)=\left(\frac{\partial
P(z)}{\partial z_1}, \ldots, \frac{\partial P(z)}{\partial
z_N}\right).$$

 For the coordinate free form,  let $u\in\mathbb C^N$ and $p$ be a polynomial in $\mathbb C^N$. We define the complex Dunkl operator by
$$ \mathcal{D}_u p(z)=\sum_{j=1}^N \overline{ u_j}\mathcal{D}_j p(z),$$
 i.e.,
\begin{eqnarray}\label{eq3.3}
 \mathcal{D}_u p(z)&=& \langle\nabla p(z), u\rangle + \sum_{v\in R_+}
\kappa_v \frac{p(z)-p(z\sigma_v)}{\langle z, v \rangle}\langle v,
u\rangle.
\end{eqnarray}

\section{Some lemmas}

For the proof of our main theorem, we need some lemmas.

 The collection of all polynomials in $z$ is
denoted by $\mathbb C[z_1,\ldots,z_N]$. We shall also use the
abbreviation
$$\Pi^N:=\mathbb C[z_1,\ldots,z_N].$$

The right regular representation of $U(N)$ is the homomorphism
\begin{eqnarray*}
\mathcal R: U(N)&\longrightarrow {\mbox {End}}\ {\Pi^N}
\\ w&\longmapsto R(w)
\end{eqnarray*}
 given by
$$\mathcal R(w)p(z)=p(zw)$$ for all $z\in \mathbb C^N$ and $p\in \Pi^N$.

\begin{lemma}
For any $u\in\mathbb C^N$, $g\in  G$, $p\in\Pi^N$,  and $z\in\mathbb
C^N$, we have
$$\mathcal R(g)^{-1}\mathcal{D}_u
\mathcal R(g)p(z)=\mathcal{D}_{ug}p(z).$$
\end{lemma}

\begin{proof}
The complex Dunkl operator can be written as the sum over the root
system $R$ instead of $R_+$ by
\begin{eqnarray*}
 \mathcal{D}_u p(z)&=& \langle\nabla p(z), u\rangle + \frac{1}{2}\sum_{v\in R}
\kappa_v \frac{p(z)-p(z\sigma_v)}{\langle z, v\rangle} \langle v,
u\rangle.
\end{eqnarray*}
By definition
\begin{eqnarray*}
z\sigma_y g = (z\sigma_y) g= zg-2\frac{\langle z, y\rangle}{||y||^2}
yg=(zg)\sigma_{yg},
\end{eqnarray*}
which means
$$ \sigma_{y g}=g^{-1} \sigma_y g.$$
Therefore
$$v=yg \Longrightarrow g\sigma_{v}= \sigma_y g.
$$

Consequently,
\begin{eqnarray*}
\mathcal R(g)  \mathcal{D}_{ug} p(z)&=&\langle\nabla p(zg),
ug\rangle + \frac{1}{2}\sum_{v\in R} \kappa_v
\frac{p(zg)-p(zg\sigma_v)}{\langle zg, v\rangle} \langle v,
ug\rangle
\\
&=&\langle\nabla p(zg)g^{-1}, u\rangle + \frac{1}{2}\sum_{y\in R}
\kappa_y \frac{p(zg)-p(z\sigma_y g)}{\langle z, y\rangle} \langle y,
u\rangle
\\
&=&  \mathcal{D}_{u} (\mathcal R(g) p(z)).\end{eqnarray*} The last
step used the fact that
$$\nabla (p(zg))=(\nabla p(zg)) g.$$
 \fp
\end{proof}

For  $v\in\mathbb R^N\setminus\{0\}$, we  define the operator
$\rho_v$ by
$$
\rho_v f(x)=\frac{f(x)-f(x\sigma_v)}{\langle x, v \rangle}.
$$

\begin{lemma}\label{le:3.2} For any $u\in\mathbb C^N$ and\  $v\in\mathbb R^N\setminus
\{0\}$,
\begin{eqnarray*}
\langle\nabla, u\rangle\rho_v f(z)-\rho_v\langle\nabla, u\rangle
f(z) =\frac{\langle v, u\rangle}{\langle z, u\rangle}
\left(\frac{2\langle\nabla f(z\sigma_v), v\rangle}{\langle v,
v\rangle}-\frac{f(z)-f(z\sigma_v)}{\langle z, v\rangle}\right).
\end{eqnarray*}
\end{lemma}

\begin{proof}
Notice that
$$
\langle \nabla R(\sigma_v) f,  u\rangle =
 \langle R(\sigma_v) \nabla
f \sigma_v^{-1},  u\rangle
=
 \langle R(\sigma_v) \nabla
f,  u \sigma_v\rangle.$$ Therefore
\begin{eqnarray*}
\langle \nabla,  u\rangle \rho_v f(z) &=&\frac{\langle \nabla
 f(z),  u\rangle-\langle \nabla  f(z\sigma_v),  u\sigma_v\rangle}{\langle z,
 v\rangle}\\
 &&\qquad -
\frac{\langle \nabla
 f(z),  u\rangle-(f(z)-f(z\sigma_v)) \langle v, u\rangle}{\langle z,
 v\rangle^2}\end{eqnarray*}
 and
 $$\rho_v \langle \nabla
,  u\rangle  f(z) = \frac{\langle \nabla
 f(z),  u\rangle-\langle \nabla  f(z\sigma_v),  u\rangle}{\langle z,
 v\rangle}.
$$
By subtracting the above two identities, the result follows from the
fact that
$$u-u\sigma_v=2
\frac{\langle u, v\rangle}{\langle v, v\rangle}.$$\fp
\end{proof}

\section{Main theorem}

Our main result gives the commutativity of the complex Dunkl
operators.

\begin{theorem} Let $R\subset \mathbb C^N$ be a root system satisfying Assumption A and let
the associated  complex Dunkl operators be
defined as in (\ref{eq3.3}). Then
$$\mathcal D_t \mathcal
D_u f(z)=\mathcal D_u \mathcal D_t f(z)$$ for any $f\in\Pi^N$ and
 $z, t, u\in\mathbb C^N$.
\end{theorem}

\begin{proof}
By definition
\begin{eqnarray*}
\mathcal{D}_t f(z) &=\langle\nabla , t \rangle f(z)+\sum_{v\in R_+}
\kappa_v \langle v, t \rangle \rho_v f(z);
\\
\mathcal{D}_u f(z) &=\langle\nabla f(z), u \rangle+\sum_{v\in R_+}
\kappa_v \langle v, u \rangle \rho_v f(z). \end{eqnarray*} Therefore
\begin{eqnarray*}
\mathcal{D}_t \mathcal{D}_u f(z)&=&\langle\nabla, t \rangle
\mathcal{D}_u f(z) +\sum_{v\in R_+} \kappa_v \langle v, t \rangle
\rho_v \mathcal{D}_u f(z)
 \\
 &=&
\langle\nabla, t \rangle \left(\langle\nabla, u \rangle f(z)+
\sum_{v\in R_+} \kappa_v \langle v, u \rangle \rho_v f(z)\right)
\\
& & +\sum_{v\in R_+} \kappa_v \langle v, t \rangle \rho_v \left(
\langle\nabla, u \rangle  f(z)+\sum_{v\in R_+} \kappa_v \langle v, u
\rangle \rho_v f(z)\right)
\\
 &=&
\langle\nabla, t \rangle  \langle\nabla, u \rangle f(z) + \sum_{v\in
R_+} \kappa_v \langle v, u \rangle \langle \nabla, t \rangle \rho_v
f(z)
\\
& &  +\sum_{v\in R_+} \kappa_v  \langle v, t \rangle \rho_v \langle
\nabla, u \rangle  f(z) +\sum_{v,s\in R_+} \kappa_v  \kappa_s\langle
v, t \rangle \langle s, u \rangle\rho_v\rho_s  f(z)
\end{eqnarray*}

By symmetricity and notice that $$\langle\nabla, t \rangle
\langle\nabla, u \rangle f(z)=\langle\nabla, u \rangle
\langle\nabla, t \rangle f(z),$$ we obtain
\begin{eqnarray*}
& &\mathcal{D}_t \mathcal{D}_u f(z)-\mathcal{D}_u \mathcal{D}_t f(z)
\\ &=&\sum_{v\in R_+} \kappa_v \displaystyle\left\{\langle v, u \rangle
\left(\langle \nabla, t \rangle \rho_v-\rho_v\langle \nabla, t
\rangle \right)- \langle v, t \rangle \left(\langle \nabla, u
\rangle\rho_v -\rho_v\langle \nabla, u \rangle \right)
\displaystyle\right\}f(z)
\\
& & + \sum_{v,s\in R_+} \kappa_v  \kappa_s \left(\langle v, t
\rangle \langle s, u \rangle-\langle v, u \rangle \langle s, t
\rangle \right) \rho_v\rho_s f(z).
\end{eqnarray*}

The first summand vanishes since Lemma \ref{le:3.2} implies
\begin{eqnarray*}  & &\langle v, u \rangle
\left(\langle \nabla, t \rangle \rho_v-\rho_v\langle \nabla, t
\rangle \right) f(z)
\\
& &\qquad =\langle v, u \rangle \frac{\langle v, t\rangle}{\langle
z, v\rangle} \left(\frac{2\langle\nabla f(z\sigma_v),
v\rangle}{\langle v, v\rangle}-\frac{f(z)-f(z\sigma_v)}{\langle z,
v\rangle}\right)
\\
& &\qquad =\langle v, t \rangle \left(\langle \nabla, u
\rangle\rho_v-\rho_v\langle \nabla, u \rangle  \right) \displaystyle
f(z).
\end{eqnarray*}

To estimate the second summand, we denote $$ B(v, s)=\langle v, t
\rangle \langle s, u \rangle-\langle v, u \rangle \langle s, t
\rangle, \qquad \forall\ v, s\in\mathbb R^N.$$ It remains to show
\begin{eqnarray} \label{eq4.4}\sum_{v,s\in R_+} \kappa_v  \kappa_s B(v, s) \rho_v\rho_s
f(z)=0.\end{eqnarray}

Notice that $B(v, s)$ satisfies the following properties:

(i) $B(v, s)$ is an antisymmetric bilinear form.

(ii) $B(v \sigma_y, s\sigma_y)=B(s, v), \qquad \forall\ y=av+bs,
\quad a, b\in \mathbb R, \quad v, s\in R_+.$

Indeed, since $B(v,v)=B(s,s)=B(y,y)=0$ by antisymmetry, we have
\begin{eqnarray*}
B(v \sigma_y, s\sigma_y)&=&B(v-2\frac{\langle v, y\rangle}{||y||^2}
y, s-2\frac{\langle s, y\rangle}{||y||^2} y)
\\
&=&B(v,s)-2\frac{\langle s, y\rangle}{||y||^2}  b B(v, s)
-2\frac{\langle v, y\rangle}{||y||^2} a B(v, s)
\\
&=&-B(v,s)=B(s,v).
\end{eqnarray*}
These property imply  (\ref{eq4.4}) due  to Corollary  4.4.7 in
\cite{dunklx}. This completes the proof. \fp
\end{proof}

\begin{Example} In $\mathbb C^{1}$,  the root system $R=A_1$ is the unique system satisfying Assumption A. In this case, the reflection group $G=\mathbb Z_2$, and the multiplicity
function is given by a single parameter $\kappa\in\mathbb C$. The
complex Dunkl operator is  given  by
$$
\begin{array}{lcl}
\mathcal
Df(z)&=&f'(z)+\kappa\displaystyle\frac{f(z)-f(-z)}{z},\qquad z\in
\mathbb C.
\end{array}$$
\end{Example}

Due to the commutativity of the complex Dunkl operators, we can
introduce the related complex Dunkl Laplacian. Let $\Delta$ and
$\nabla$ be the complex Laplacian and gradient operator
$$\begin{array}{rcl}
\Delta_{\mathbb C} &=&\displaystyle\frac{\partial^2 }{\partial
z_1^2}+\cdots+ \frac{\partial^2 }{\partial z_n^2},
\\ \nabla&=&\left(\displaystyle\frac{\partial }{\partial z_1},\cdots, \frac{\partial }{\partial z_n}\right).
\end{array}
$$
The complex Dunkl Laplacian is defined as
$$\Delta_h=\mathcal D_1^2+\ldots+\mathcal D_N^2.$$
Similar to the real case in \cite{dunklx}, we have the formula
$$
\Delta_h f(z)=\Delta_{\mathbb C} f(z)+2\sum_{v\in R_+}
\kappa_v\displaystyle\frac{\langle \nabla f(z), v\rangle}{\langle z,
v\rangle}-2\sum_{v\in R_+} \kappa_v
\displaystyle\frac{f(z)-f(\sigma_v z)}{\langle z, v\rangle^2} |v|^2.
$$

\bigskip

\noindent \textbf{Acknowledgements} The research is partially
supported by the \textit{Unidade de Investiga\c c\~ao ``Matem\'atica
e Aplica\c c\~oes''} of University of Aveiro, and by the NNSF  of
China  (No. 10771201).

\end{document}